\documentclass[reqno, 11pt]{amsart}

\makeatletter
\let\origsection=\section \def\section{\@ifstar{\origsection*}{\mysection}} 
\def\mysection{\@startsection{section}{1}\z@{.7\linespacing\@plus\linespacing}{.5\linespacing}{\normalfont\scshape\centering\S}}
\makeatother  

\usepackage{mathrsfs}
\usepackage{amsmath}
\usepackage{amssymb}
\usepackage{amsthm}
\usepackage{mathtools}
\usepackage{letterswitharrows}
\usepackage{tikz-cd}
\usetikzlibrary{intersections,spath3}
\usepackage{thmtools}
\usepackage{thm-restate}
\usepackage[only,llbracket,rrbracket]{stmaryrd}
\usepackage[linesnumbered,ruled,vlined]{algorithm2e}
\usepackage{enumitem}
\setenumerate{label={\normalfont (\roman*)}}

\usepackage[utf8]{inputenc}
\usepackage[T1]{fontenc}
\usepackage{lmodern}
\usepackage[babel]{microtype}
\usepackage[english]{babel}
\usepackage{relsize}

\usepackage{graphicx}
\usepackage{subcaption}
\usepackage{import}

\usepackage{svg}
\usepackage[extractname=filename]{svg-extract}

\usepackage{geometry}
\usepackage{xcolor} 
\colorlet{darkishRed}{red!80!black}
\colorlet{darkishBlue}{blue!60!black}
\colorlet{darkishGreen}{green!60!black}
\definecolor{lucasBlue}{HTML}{A4D4F5}
\definecolor{lucasDarkGreen}{HTML}{8CBF65}
\definecolor{lucasLightRed}{HTML}{F5B3B9}
\definecolor{lucasOrange}{HTML}{EBC07A}

\usepackage{hyperref}
\hypersetup{
	colorlinks,
	linkcolor={blue!60!black},
	citecolor={green!60!black},
	urlcolor={blue!60!black},
}
\usepackage[abbrev, msc-links]{amsrefs}
\usepackage[nameinlink, capitalise, noabbrev]{cleveref}
\crefformat{enumi}{#2#1#3}
\crefformat{equation}{#2(#1)#3}
\crefname{mainresult}{Theorem}{Theorems}
\usepackage{doi}

\renewcommand{\PrintDOI}[1]{\doi{#1}}

\let\setminus=\smallsetminus

\renewcommand{\leq}{\leqslant}
\renewcommand{\geq}{\geqslant}

\newtheorem{theorem}{Theorem}[section] 

\newtheorem{corollary}[theorem]{Corollary}
\newtheorem{lemma}[theorem]{Lemma}

\newcounter{claimcounter}[theorem]
\newtheorem{claim}[claimcounter]{Claim}
\newtheorem*{claim*}{Claim}
\crefname{claim}{Claim}{Claims}
\newcounter{casecounter}[theorem]
\newtheorem*{case*}{Case}
\crefname{case}{Case}{Cases}

\newcounter{subclaimcounter}[claimcounter]
\newtheorem*{subclaim*}{Subclaim}

\theoremstyle{definition}

\crefname{mainexample}{Example}{Examples}

\crefname{example}{Example}{Examples}

\crefname{routine}{Routine}{Routines}

\crefname{subroutine}{Subroutine}{Subroutines}

\crefname{subsubroutine}{Subsubroutine}{Subsubroutines}

\crefname{step}{Step}{Steps}
\theoremstyle{remark}

\newcommand{\COMMENT}[1]{{}}

\let\epsilon=\varepsilon
\let\theta=\vartheta
\let\rho=\varrho
\let\phi=\varphi
\def\N{\mathbb N}

\makeatletter

\def\calCommandfactory#1{%
  \expandafter\def\csname c#1\endcsname{\mathcal{#1}}}
\def\frakCommandfactory#1{%
  \expandafter\def\csname frak#1\endcsname{\mathfrak{#1}}}
   
\newcounter{ctr}
\loop
  \stepcounter{ctr}
  \edef\X{\@Alph\c@ctr}
  \expandafter\calCommandfactory\X
  \expandafter\frakCommandfactory\X
\ifnum\thectr<26
\repeat

\usepackage{etoolbox}

\newbool{arXiv}
\booltrue{arXiv} 

\newcommand{\arXivOrNot}[2]{\ifbool{arXiv}{{#1}}{{#2}}}

\newbool{pdfBool}
\booltrue{pdfBool} 

\newcommand{\pdfOrNot}[2]{\ifbool{pdfBool}{{#1}}{{#2}}}

\makeatletter
\newcommand\thankssymb[1]{\textsuperscript{\@fnsymbol{#1}}}
\makeatother

\newcounter{mylabelcounter}

\makeatletter
\newcommand{\labelText}[2]{%
#1\refstepcounter{mylabelcounter}%
\immediate\write\@auxout{%
  \string\newlabel{#2}{{1}{\thepage}{{\unexpanded{#1}}}{mylabelcounter.\number\value{mylabelcounter}}{}}%
}%
}

\makeatletter
\newcommand\footnoteref[1]{\protected@xdef\@thefnmark{\ref{#1}}\@footnotemark}
\makeatother

\definecolor{cMaroon}{HTML}{93152a}
\newcommand{\defn}[1]{{\color{cMaroon}{\emph{#1}}}}

\usepackage{tabularx}

\title{On order-compatible paths in infinite graphs}

\author[M.~Pitz]{Max Pitz}

\author[L.~Real]{Lucas Real\thankssymb{1}}
\thanks{\thankssymb{1} Supported by São Paulo Research Foundation (FAPESP) through grant number 2025/00669-5.}

\author[R.~Schaut]{Roman Schaut\thankssymb{2}}
\thanks{\thankssymb{2} Supported by a State Graduate Funding Program scholarship of the University of Hamburg.}

\address{University of Hamburg, Department of Mathematics, Bundesstrasse 55 (Geomatikum), 20146 Hamburg, Germany}
\email{\{max.pitz, roman.schaut, lucas.real\}@uni-hamburg.de}

\begin{document}

\begin{abstract}
    Two $a{-}b$ paths in a graph $G$ are \emph{order-compatible} if their common vertices occur in the same order when travelling from $a$ to $b$.
    Suppose a graph contains an infinite number $\delta$ of edge-disjoint $a{-}b$ paths. 
    G.A.~Dirac asked whether there always exists a family of $\delta$ edge-disjoint $a{-}b$ paths that are pairwise order-compatible. Confirming a conjecture by B.~Zelinka, we show that this holds provided that the given $\delta$ edge-disjoint $a{-}b$ paths have bounded length. 
    Combining this with an earlier work of Zelinka, it follows that Dirac's question for an infinite cardinal $\delta$ has an affirmative answer if and only if $\delta$ has uncountable cofinality.

    As our second main result, we show that even when Dirac's question fails, it  still holds that `being connected by $\delta$ edge-disjoint, pairwise order-compatible paths' is an equivalence relation for all values of $\delta$. The most interesting case here is when $\delta$ is countable. 
\end{abstract}



\vspace*{-12pt}
\maketitle

\section{Introduction}

This paper is concerned with the following question of  G.\,A.~Dirac:

\begin{quote}
\emph{
The two vertices $a$ and $b$ of a graph are connected by an infinite number $\delta$ of paths each of which have $a$ and $b$ as their end-vertices and no pair of which have an edge in common. Are $a$ and $b$ necessarily connected by $\delta$ paths such that each of them has $a$ and $b$ as its end-vertices, no two have an edge in common, and the common vertices of any two of the paths always occur in the same order along both going from $a$ to $b$? The graphs may contain multiple edges.
}
\hfill \cite[Problem 6, pp.~158--159]{fiedler1964theory}
\end{quote}

Let us introduce some notation related to Dirac's question. Two $a{-}b$ paths in a graph\footnote{As in the setting of Dirac's question, graphs in this paper may have parallel edges but no loops.} $G$ are \defn{order-compatible} if both traverse their common vertices in the same order. For a (finite or infinite) cardinal $\delta$ and a given graph $G$, we write
\defn{$a \sim_\delta b$} if $G$ contains $\delta$ edge-disjoint $a{-}b$ paths, and \defn{$a \approx_\delta b$} if $G$ contains $\delta$ edge-disjoint $a{-}b$ paths that are pairwise order-compatible. Essentially, Dirac wanted to know whether $a \sim_\delta b$ always implies $a \approx_\delta b$, 
which we will refer to as \emph{Dirac's question for the cardinal $\delta$}. Given the recent development regarding edge-connectivity in infinite graphs \cites{kurkofka2022every, knappe2024immersion}, we believe there is a renewed interest for a complete solution to Dirac's question, which is the purpose of this paper.

That Dirac's question is true for finite $\delta$ is an easy exercise: simply take any collection of edge-disjoint $a{-}b$ paths $P_1,\ldots,P_k$ such that the cardinality $|E(P_1) \cup \cdots \cup E(P_k)|$ is minimal.
However, B.~Zelinka established, perhaps unexpectedly, that the answer to Dirac's question is in the negative for $\delta = \aleph_0$ \cite{zelinka1967poznamka}\,\footnote{A different example of an infinitely edge-connected, countable, planar graph in which no two vertices are connected by infinitely many edge-disjoint, pairwise order-compatible paths has been  given by J.~Kurkofka \cite{WhirlGraph}.}, and extended his counterexample in \cite{zelinka1968nespovcetne} to all cardinals $\delta$ of countable cofinality.
On the other hand, Zelinka showed in \cite{zelinka1968nespovcetne} that Dirac's question has an affirmative answer when $\delta$ is regular uncountable, or when $\delta = \aleph_0$ and $a$ is connected to $b$ by infinitely many paths of bounded length $\leq p$ for some $p \in \N$. Based on the last result, he conjectured \cite[Hypothesis, p.~254]{zelinka1968nespovcetne} that this last obstacle is the key to solving Dirac's question. 

Our first main result  confirms Zelinka's conjecture, and resolves the remaining open cases of Dirac's question.

\begin{restatable}{theorem}{ZelinkasConjecture}
\label{thm_zelinaconj}
     Let $\delta$ be an infinite cardinal and $a,b \in V(G)$ be two vertices for which there are $\delta$ edge-disjoint $a{-}b$ paths of length at most $p$, for some $p \in \mathbb{N}$. 
    Then there exist $\delta$ edge-disjoint  $a{-}b$ paths that are pairwise order-compatible.
\end{restatable}

\begin{corollary}
    \label{cor_main}
    Dirac's question for an infinite cardinal $\delta$ is true precisely when $\delta$ has uncountable cofinality.
\end{corollary}

\begin{proof}[Proof of \cref{cor_main}]
    We already mentioned that the answer to Dirac's question is in the negative if $\delta$ has countable cofinality. 
    Conversely, let $\delta$ be an infinite cardinal of uncountable cofinality, and $a,b \in V(G)$ be two vertices admitting a $\delta$ sized family $\cP$ of edge-disjoint $a{-}b$ paths. 
    For $n \in \N$, let $\cP_n$ be the subfamily of $\cP$ consisting of all paths of length $n$. 
    Since $\sum_{n \in \N} |\cP_n| =\delta$ and $\delta$ has uncountable cofinality, there is some $p \in \N$ with $|\cP_p| = \delta$. Applying \cref{thm_zelinaconj} to this path family $\cP_p$ finishes the proof.
\end{proof}

But our paper is not only about uncountable combinatorics. Indeed, our second main result is  especially interesting in the countable case.   
To appreciate it, recall that Menger's theorem for infinite graphs readily implies that  $\sim_\delta$ is an equivalence relation on $V(G)$. Thus, a consequence of Dirac's question -- if positive -- would have been that $\approx_\delta$ is an equivalence relation, too. 
Now, in contrast to the situation regarding Dirac's question, we prove that this weaker consequence holds for all cardinals $\delta$ (but the equivalence classes under $\approx_\delta$ may be finer than the classes under $\sim_\delta$, when $\delta$ has countable cofinality).

\begin{restatable}{theorem}{Transitive}
\label{thm_main2}
For any graph $G$ and for any finite or infinite cardinal $\delta$, the relation $\approx_\delta$ is an equivalence relation on $V(G)$.
\end{restatable}

This paper is organised as follows.
We begin in \cref{sec_reg} with a technical strengthening of Zelinka's result from \cite{zelinka1968nespovcetne} when $\delta$ is regular infinite. This technical strengthening in the regular case will be the key ingredient that allows us to lift the proof to the singular case, thus completing the proof of Zelinka's conjecture, i.e., of \cref{thm_zelinaconj}.

We then proceed to prove \cref{thm_main2}. The crucial and most interesting case in \cref{thm_main2} is the countable case $\delta = \aleph_0$, which we will prove in \cref{sec_ordercomp_cble}. 
In \cref{sec_ordercomp_uncble}, we complete the proof of \cref{thm_main2} by lifting our construction to all singular cardinals.

\section{Zelinka's conjecture}
\label{sec_reg}

All graphs in this paper may have parallel edges.
Given two distinct vertices $a,b$ in a graph $G$, let \defn{$\kappa_V(a,b)$} denote the maximal size of an $a{-}b$ path system of pairwise internally vertex disjoint paths, and let \defn{$\kappa_E(a,b)$} denote the maximal size of an $a{-}b$ path system of pairwise edge-disjoint paths.
The \defn{length} of a path is its number of edges. For a given path $P$ in $G$ and two vertices $u,v\in P$, the unique subpath in $P$ from $u$ to $v$ is denoted by \defn{$uPv$}.
Given $n \in \mathbb{N}_0$ we also write $[n]:=\{0, 1, \dots, n\}$.

\begin{lemma} \label{lem:regularCardCase}
    Let $\delta$ be a regular infinite cardinal\,\footnote{So $\delta$ is either $\omega$, or regular uncountable.} and $a,b \in V(G)$ two vertices admitting a family $\mathcal{P}$ of size $\delta$ of edge-disjoint paths of length $\leq p$, for some $p \in \mathbb{N}$.
    
    Then there exists a sequence $Q \colon a=t_0 t_1 \dots t_k=b$ of distinct vertices in $\mathcal{P}$ with $\kappa_V(t_i, t_{i+1})\geq \delta$ for every $i \in [k-1]$.
    
    Moreover, there exists a subfamily $\cP'\subseteq \mathcal{P}$ such that every path $P \in \cP'$ traverses every vertex on $Q$ and exactly in the order of $Q$,\footnote{But they may intersect outside of $Q$ and so are not necessarily order-compatible.}
    implying that $k\leq p$.
\end{lemma}

\begin{proof}
    We proceed by induction on $p$.
    If $p = 1$, then any two paths in $\cP$ are internally vertex-disjoint and hence $\kappa_V(a,b) = \delta$.
    Moreover, if $\kappa_V(a,b) \geq \delta$, the sequence $Q\colon ab$ together with the entire path system $\mathcal{P}$ trivially verifies the statement.
    
    Hence, we might assume $p\geq 2$ and $\kappa_V(a,b)<\delta$.
    In particular, we can fix a corresponding set $S\subseteq V(G)\setminus \{a,b\}$ with $|S|<\delta$ and which is an $a{-}b$ separator, thus meeting every path in $\cP$.
    Since $\delta$ is regular and infinite, some $s \in S$ must hit $\delta$ many paths from $\mathcal{P}$.
    Let $\cP_0 \subseteq \mathcal{P}$ be the subfamily consisting of those paths.
    By induction, and regarding the $a{-}s$ path family $(aPs)_{P \in \cP_0}$, we find a sequence $Q_1 \colon a=t_0 t_1\dots t_{k_1}=s$ with  $\kappa_V(t_i, t_{i+1}) \geq \delta$ for every $i \in [k_1-1]$, as well as a suitable subfamily $\cP_1$ of $\cP_0$ of size $\delta$ that satisfies the `moreover' part with respect to $Q_1$.
   Similarly, applying the induction assumption to $(sPb)_{P \in \cP_1}$, we find a sequence $Q_2 \colon s=l_0 l_1\dots l_{k_2}=b$ such that $\kappa_V(l_i, l_{i+1}) \geq \delta$ for every $i \in [k_2 -1]$, and a suitable subfamily $\cP_2 \subseteq \cP_1$ of size $\delta$ that satisfies the 'moreover' part with respect to $Q_2$.

    Also, $Q_1$ and $Q_2$ intersect only in the vertex $s$.
    Indeed, every path $P \in \cP_2$ first traverses every vertex of $Q_1$ on the subpath $aP s$ and then every vertex of $Q_2$ on the subpath $sP b$.
    Hence, the concatenation of the two sequences ${Q_1}^\frown Q_2$ and the corresponding system $\cP_2$ satisfy the lemma and the `moreover' part of the statement.
\end{proof}

With the help of the previous lemma, we now prove our first main theorem, which we recall here for convenience.

\ZelinkasConjecture*

\begin{proof} We divide the proof into two cases, depending on whether $\delta$ is regular or singular.

\smallskip
\textbf{The regular case.} If $\delta$ is regular, then by \cref{lem:regularCardCase} there exists $t_0,t_1,\dots,t_n$ with $a=t_0$, $b=t_n$ and such that $\kappa_V(t_i,t_{i+1})\geq \delta$ for every $i \in [n-1]$.
    In other words, there exists a family $\mathcal{Q}_i$ of (internally) disjoint $t_i$-$t_{i+1}$ paths, and $\mathcal{Q}_i$ has size $\delta$.
    Then, let $P_0$ be a concatenation of the form $P_0:=Q_0^0Q_1^0Q_2^0\dots Q_{n-1}^0$, where any two paths $Q_i^0 \in \mathcal{Q}_i$ and $Q_j^0 \in \mathcal{Q}_j$ are internally disjoint for every $0 \leq i < j < n$.
    For some ordinal $\alpha<\delta$, suppose that we have defined a family of edge-disjoint paths $(P_{\beta}=Q_0^{\beta}Q_1^{\beta}\dots Q_{n-1}^{\beta})_{\beta <\alpha}$ such that $Q_i^{\gamma}\in \mathcal{Q}_i$ is internally disjoint from $Q_{j}^{\gamma'}\in\mathcal{Q}_j$ for every $0\leq i,j<n $ and $\gamma,\gamma'<\alpha$. 
    Since $\alpha < \delta$ and $\delta$ is regular, the set of vertices $X_{\alpha}:=\bigcup_{\beta <\alpha}V(P_\beta)$ has size $<\delta$.
    In particular, we may fix $Q_0^{\alpha}$ as a $t_0-t_1$ path internally disjoint from $X_{\alpha}$.
    Similarly, for $1\leq i <n$, we choose $Q_i^{\alpha}\in \mathcal{Q}_i$ to be internally disjoint from $X_{\alpha}\cup \bigcup_{j<i}V(Q_j^{\alpha})$.
    We finish this recursive process after setting $P_{\alpha} = Q_0^{\alpha}Q_{1}^{\alpha}\dots Q_{n-1}^{\alpha}$.
    Then, $(P_{\alpha})_{\alpha <\delta}$ is a family of edge-disjoint $a{-}b$ paths such that, for $\alpha < \beta< \delta$, the paths $P_{\alpha}$ and $P_{\beta}$ intersect on $t_0,t_1,t_2,\dots,t_n$ and precisely in that order.

 \smallskip
\textbf{The singular case.} If $\delta$ is singular, then fix an increasing sequence $(\delta_{\alpha})_{\alpha < \operatorname{cf}(\delta) }$ of regular cardinals converging towards $\delta$ and so that $\operatorname{cf}(\delta) < \delta_{0}$.
    Let $H_{\alpha}$ be the graph with vertex set $V(G)$ and edge set $\{ab \colon \kappa_V(a,b)\geq \delta_{\alpha} \}$, for every ${\alpha}< \operatorname{cf}(\delta)$.
    Since each $\delta_{\alpha}$ is regular, by the `Moreover'-part in \cref{lem:regularCardCase} we find a family $(P_{\alpha})_{{\alpha}<\operatorname{cf}(\delta)}$ where each $P_{\alpha}$ is an $a{-}b$ path in $H_{\alpha}$ of length $\leq p$.

    Let $H$ be the edge-disjoint union of all paths $P_{\alpha}$. 
    Formally, we define a multigraph $H$ with vertex set $V(H)=V(G)$ and edges $E(H)=\dot{\bigcup}_{{\alpha}<\operatorname{cf}(\delta)}E(P_{\alpha})$.
    By construction, there exists $\operatorname{cf}(\delta)$ pairwise edge-disjoint $a{-}b$ paths in $H$ of length $\leq p$ and hence, by the regular case of this theorem, also a family $(Q_{\alpha})_{{\alpha}<\operatorname{cf}(\delta)}$ of pairwise edge-disjoint and order-compatible $a{-}b$ paths in $H$.

    The proof will be concluded by relying on \cref{lem:auxilaryGraph} below. 
    For this, it suffices to prove that for every ${\alpha}<\operatorname{cf}(\delta)$, there exists some $\beta<\operatorname{cf}(\delta)$ such that $\kappa^G_V(u,v) \geq \delta_{
    \alpha}$, for every edge $uv \in E(Q_\beta)$.
But this is easy: given some $\alpha<\operatorname{cf}(\delta)$, the set of edges $\dot\bigcup_{\gamma<\alpha}E(P_\gamma)$ has size $<\operatorname{cf}(\delta)$.
    Thus, there exists a path $Q_\beta$ in $H$ that is edge-disjoint from $\dot\bigcup_{\gamma<\alpha}E(P_\gamma)$ and hence uses edges from paths $P_\mu$ where $\mu \geq \alpha$.
    These paths witness that $\kappa^G_V(u,v)\geq \delta_\alpha$ for every edge $uv \in E(Q_\beta)$.
\end{proof}

\begin{lemma}\label{lem:auxilaryGraph}
    Fix a graph $G$. Let $\delta$ be a singular infinite cardinal, and $(\delta_{\alpha})_{\alpha< \operatorname{cf}(\delta)}$ an increasing sequence of regular cardinals converging towards $\delta$ with $\operatorname{cf}(\delta)<\delta_{0}$.
    Let $H_{\alpha}$ be the graph on vertex set $V(G)$ and edge set $\{ xy \colon \kappa_V(x,y)\geq \delta_{\alpha} \}$ for every $\alpha<\operatorname{cf}(\delta)$.
    
    If there exists a choice of $a{-}b$ paths $P_{\alpha} \in H_{\alpha}$ (not necessarily distinct) such that $P_{\alpha}$ and $P_{\beta}$ are order-compatible for every $\alpha \neq \beta$, then there exist $\delta$ many edge-disjoint $a{-}b$ paths in $G$ that are pairwise order-compatible.
\end{lemma}
\begin{proof}

    We will define families $(\mathcal{Q}_\alpha)_{\alpha<\operatorname{cf}(\delta)}$ of $a{-}b$ paths inductively, such that
    \begin{enumerate}
        \item $\mathcal{Q}_\alpha$ consists of $\delta_\alpha$ many pairwise edge-disjoint and order-compatible $a{-}b$ paths,
        \item \label{item:auxilaryModel} $V(Q) \cap \bigcup_{\beta<\operatorname{cf}(\delta)}V(P_\beta)=V(P_\alpha)$ for every path $Q \in \mathcal{Q}_\alpha$ and $Q$ traverses $V(P_\alpha)$ in exactly the order induced by $P_\alpha$, and 
        \item \label{item:auxilaryOrderComp} any two paths $Q_\alpha \in \mathcal{Q}_\alpha$ and $Q_\beta \in \mathcal{Q}_\beta$ are edge-disjoint and intersect in exactly $V(Q_\alpha) \cap V(Q_\beta) = V(P_\alpha) \cap V(P_\beta)$.
    \end{enumerate}

    Note that \cref{item:auxilaryOrderComp} implies that any two paths $Q_\alpha \in \mathcal{Q}_\alpha$ and $Q_\beta \in \mathcal{Q}_\beta$ with $\alpha \neq \beta$ are edge-disjoint but also order-compatible: Their common intersection is exactly $V(P_\alpha) \cap V(P_\beta)$, and by \cref{item:auxilaryModel} they traverse their common intersection in exactly the same order as $P_\alpha$ and $P_\beta$ do.
    Since we assumed $P_\alpha$ and $P_\beta$ to be order-compatible, $Q_\alpha$ and $Q_\beta$ are also order-compatible.

     Let $\beta<\operatorname{cf}(\delta)$ and assume we have defined $\mathcal{Q}_\alpha$ for every $\alpha<\beta$.
    Let $$S:= \bigg(\bigcup_{\gamma<\beta}\bigcup_{Q \in \mathcal{Q}_\gamma} (V(Q) \cup E(Q)) \cup \bigcup_{\gamma<\operatorname{cf}(\delta)} V(P_\gamma) \bigg) \setminus V(P_{\beta})$$

    We have $|S|\leq \sup \{ \delta_{\alpha} \colon \alpha < \beta \} \cdot \omega + \operatorname{cf}(\delta) \cdot \omega < \delta_{\beta}$, where in the last inequality we used the fact that $\beta < \operatorname{cf}(\delta) < \delta_{\beta}$ and $\delta_{\beta}$ is regular.
        
    It holds that $|S|<\delta_{\beta}$ and, since $P_{\beta} \in H_{\beta}$, any two adjacent vertices on $P_{\beta}$ have vertex connectivity at least $\delta_{\beta}$ in $G$ (and in $G-S$ as well).
    Thus, as in the proof of the regular case of \cref{thm_zelinaconj},     there exist $\delta_{\beta}$ many pairwise edge-disjoint and order-compatible $a{-}b$ paths in $G-S$ satisfying \cref{item:auxilaryModel}, which we pick as $\mathcal{Q}_{\beta}$. 
    By definition of $S$, any two paths $Q_\alpha \in \mathcal{Q}_\alpha$ and $Q_{\beta} \in \mathcal{Q}_{\beta}$, for some $\alpha<\beta$, intersect in exactly $V(P_\alpha) \cap V(P_{\beta})$ and are edge-disjoint.
    Hence, $\mathcal{Q}_{\beta}$ satisfies \cref{item:auxilaryOrderComp}.
    
    This finishes the construction of $\cQ_{\beta}$ and the lemma follows.
\end{proof}

\section{Order compatibility is transitive - the countable case}
\label{sec_ordercomp_cble}

In this section we show that $\approx_\omega$ is a transitive relation, thus establishing the countable case of Theorem~\ref{thm_main2}.
To that aim, we begin with a useful lemma on concatenating paths:

\begin{lemma}\label{lem:Concatenation}
    Let $a,u,v,c\in V(G)$ be vertices. 
    Fix two pairwise edge-disjoint and order-compatible paths $aPu$ and $aP'v$, as well as two pairwise edge-disjoint and order-compatible paths $uQc$ and $vQ'c$.
    If $(aPu)\cap (uQc) = \{u\}$ and $(aP'v)\cap(vQ'c) = \{v\}$, then the concatenations $X:=aPuQc$ and $X':=aP'vQ'c$ are distinct $a{-}c$ paths.
    Moreover, if $(aPu)\cap (vQ'c)\subseteq \{u\}$ and $(aP'v)\cap (uQc)\subseteq \{v\}$, then $X$ and $X'$ are edge-disjoint and order-compatible.
\end{lemma}
\begin{proof}
    Since $(aPu)\cap (uQc) = \{u\}$ and $(aP'v)\cap(vQ'c) = \{v\}$, it is clear that both $X$ and $X'$ are $a{-}c$ paths, which we assume to be oriented from $a$ to $c$.
    Moreover, since each of the intersections $(aPu)\cap (vQ'c)$ and $(aP'v)\cap (uQc)$ contains at most one vertex, the pairs $\{aPu,vQ'c\}$ and $\{aP'v,uQc\}$ are edge-disjoint.
    Hence, the edge-disjointness of $X$ and $X'$ now follows from the edge-disjointness assumption over $\{aPu,aP'v\}$ and $\{uQc,vQ'c\}$.

    What remains is to prove that $X$ and $X'$ are order-compatible.
    Towards this goal, let us observe the following two claims.
    \begin{claim}\label{FirstCaseDistinction}
        Fix $x\in X\cap X'\setminus \{u,v\}$.
        If $x\in (aPu)$ (resp. $x\in (uQc)$), then $x\in (aP'u)$ (resp. $x\in (vQ'c)$) as well.
    \end{claim}
    \begin{proof}\renewcommand{\qedsymbol}{$\blacksquare$}
    After all, by definition of $X'$, we would have $x\in (aPu)\cap (vQ'c)\subseteq \{u\}$ (resp. $x\in (aP'v)\cap (uQc)\subseteq \{v\}$) in the contrary.
    \end{proof}
    \begin{claim}  \label{DistinctionUV} 
        If $u \neq v$, then $\{u, v\} \not\subseteq X \cap X'$.
    \end{claim}
    \begin{proof}\renewcommand{\qedsymbol}{$\blacksquare$}
    Suppose towards a contradiction that $u\neq v$ and $\{u,v\}\subseteq X\cap X'$.
    Then either $v\in (aPu)$ or $v\in (uQc)$.
    But the former case contradicts the assumption $u\neq v$, since it would imply $v\in (aPu)\cap (vQ'c) \subseteq \{u\}$.
    Hence, $v\in (uQc)$.
    
    Moreover, this implies $u\in (vQ'c)$: Otherwise, $u \in (aP'v)$ and thus, $u,v\in (aP'v)\cap (uQc)$, contradicting the assumption that $(aP'v)\cap (uQc) \subseteq \{ v \}$.
    Therefore, $u,v\in (vQ'c)\cap (uQc)$ with $vQ'c$ traversing first $v$ and then $u$, while $uQc$ traverses first $u$ and then $v$.
    But this contradicts the order-compatibility of $Q$ and $Q'$.
    \end{proof}

    Now, let $x \neq x' \in X \cap X'$ and, by \cref{DistinctionUV}, we can assume without loss of generality that $x\notin \{u,v\}$.
    Then either $x \in (aPu)$ or $x \in (uQc)$ due to the definition of $X$.
    Incidentally, unless by exchanging the roles of $a$ with $c$, we can further suppose that $x\in (aPu)$.
    Hence, \cref{FirstCaseDistinction} shows that $x\in (aP'v)$  as well.
    By \cref{FirstCaseDistinction}, we have the following case distinctions:
    \begin{itemize}
        \item If $x'\in (aPu)$, then $x'\in (aP'v)$ or $x'=u$.
        In the first case, it follows that $x,x' \in (aPu) \cap (aP'v)$ and, thus, $X$ and $X'$ traverse both vertices $x, x'$ in the same common order induced by the order-compatible pair $P,P'$. If $x'=u$ but $u\notin (aP'v)$, however, then $u\in (vQ'c)$ by definition of $X'$. In particular, by their own definition, both $X$ and $X'$ first traverse $x$ and then $x'$;

        \item If $x'\in (uQc)$, then $x'\in (vQ'c)$ or $x'=u$. 
        In the first case, it follows that $x \in (aPu)\cap (aP'v)$ and $x' \in (uQc) \cap (vQ'c)$ and, by their own definition, $X$ and $X'$ traverse first $x$ before $x'$.
        If $x'=u$ but $u \notin (vQ'c)$, it follows that $u \in (aP'v)$ by definition of $X'$ and hence $x,x' \in (aPu) \cap (aP'v)$ and $X$ and $X'$ traverse both vertices $x, x'$ in the same common order induced by the order-compatible pair $P,P'$. \qedhere
    \end{itemize}
\end{proof}

From now on in a fixed graph $G$, let $a,b,c\in V(G)$ be three distinct vertices such that $a\approx_{\omega}b$ and $b\approx_{\omega}c$, which means that there is
\begin{itemize}
\item  a system $\mathcal{P}_{a,b}$ of $\aleph_0$ pairwise edge-disjoint and order-compatible $a{-}b$ paths, and
\item a system $\mathcal{P}_{b,c}$ of $\aleph_0$ pairwise edge-disjoint and order-compatible $b{-}c$ paths.
\end{itemize}
We shall conclude that there is also a system $\mathcal{P}_{a,c}$ of $\aleph_0$ pairwise edge-disjoint and order-compatible paths connecting $a$ and $c$, thus concluding $a\approx_{\omega}c$.
To that aim, for two given infinite subfamilies $\mathcal{P}\subseteq \mathcal{P}_{a,b}$ and $\mathcal{Q}\subseteq \mathcal{P}_{b,c}$, we say that $v\in V(G)$ is a \defn{$(\mathcal{P},\mathcal{Q})$-terminal} if $v\in \bigcup\mathcal{P}$, $v\in Q$ for every $Q\in \mathcal{Q}$ and, moreover, $(vQc)\cap \bigcup\mathcal{P} = \{v\}$.
Considering this definition, the two criteria below provide first conditions for finding systems of edge-disjoint and order-compatible $a{-}c$ paths:

\begin{lemma}\label{lem:InfiniteHitting}
    Suppose that $\mathcal{P}\subseteq \mathcal{P}_{a,b}$ and $\mathcal{Q}\subseteq \mathcal{P}_{b,c}$ are infinite subfamilies that admit a $(\mathcal{P},\mathcal{Q})$-terminal $v$.
    If $v$ belongs to infinitely many paths in $\mathcal{P}$, then there is a system $\mathcal{P}_{a,c}$ of pairwise edge-disjoint and order-compatible $a{-}c$ paths.
\end{lemma}
\begin{proof}
    Fix an enumeration $\{Q_n\}_{n\in\mathbb{N}}$ for $\mathcal{Q}$, and an enumeration $\{P_n\}_{n\in\mathbb{N}}$ for the paths in $\mathcal{P}$ containing $v$.
    Due to the definition of a $(\mathcal{P},\mathcal{Q})$-terminal, we have $(aP_nv)\cap (vQ_mc) = \{v\}$ for every pair $n,m\in\mathbb{N}$. 
    Therefore, \cref{lem:Concatenation} shows that $\mathcal{P}_{a,c}:=\{aP_nvQ_nc\}_{n\in\mathbb{N}}$ is a system of pairwise edge-disjoint and order-compatible $a{-}c$ paths.
\end{proof}

\begin{lemma}\label{lem:NoTerminals}
    Let $\mathcal{P}\subseteq \mathcal{P}_{a,b}$ and $\mathcal{Q}\subseteq \mathcal{P}_{b,c}$ be infinite subfamilies.
     If no pair of infinite subfamilies $\mathcal{P}'\subseteq \mathcal{P}$ and $\mathcal{Q}' \subseteq \mathcal{Q}$ admits a $(\mathcal{P}',\mathcal{Q}')$-terminal, then there exists a system $\mathcal{P}_{a,c}$ of $\aleph_0$ pairwise edge-disjoint and order-compatible $a{-}c$ paths.
\end{lemma}
\begin{proof}
    Given a path $Q\in \mathcal{Q}$, we first note that $b\in P \cap Q$ for every $P\in \mathcal{P}$.
    In this case, the shortest $c{-}\bigcup\mathcal{P}$ subpath of $Q$ has the form $cQv_{Q}$ for some vertex $v_{Q}\in Q$, which then satisfies $(cQv_{Q})\cap \bigcup\mathcal{P} = \{v_Q\}$.
    
    If the set $\{v_Q\colon Q\in \mathcal{Q}\}$ is finite,
    we find an infinite subsystem $\mathcal{Q}'\subseteq \mathcal{Q}$ such that $v:=v_{Q} = v_{Q'}$ for every $Q,Q'\in \mathcal{Q}'$.
    But then $v$ is a $(\mathcal{P}, \mathcal{Q}')$-terminal contradicting the hypothesis of the lemma.
   
    Hence $\{v_Q\colon Q\in \mathcal{Q}\}$ is infinite and, after passing  to a suitable infinite subfamily of $\mathcal{Q}$, we may assume that $v_Q \neq v_{Q'}$ for distinct $Q,Q'\in \mathcal{Q}$.
    When writing $\mathcal{Q}=\{Q_n\}_{n\in\mathbb{N}}$ and $v_n:=v_{Q_n}$ for each $n\in\mathbb{N}$, we can even fix a path $P_n\in\mathcal{P}$ containing $v_n$: after all, 
    $v_n \in \bigcup \cP$ by the choice of $v_{Q_n}$.  


\begin{figure}[htbp]
    \centering
    \includegraphics[width=0.7\linewidth]{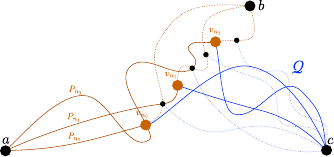}
    \caption{Choosing the paths $P_{n_0  },P_{n_1},P_{n_2},\ldots$ as in the proof of \cref{lem:NoTerminals}.}
    \label{fig:noterminalsTmp}
\end{figure}

    Next, set $n_0:=0$ and, for some $k\in\mathbb{N}$, suppose that we have defined indexes $n_0<n_1<\dots <n_k$ so that the paths $P_{n_0},P_{n_1},\dots,P_{n_k}$ are distinct. 
    Since $\{v_n\}_{n\in\mathbb{N}}$ comprises infinitely many distinct vertices and $\{P_{n_0},P_{n_1},\dots,P_{n_k}\}$ is a set of finitely many paths, we pick $n_{k+1}>n_k$ so that $v_{n_{k+1}}\notin \bigcup_{i\leq k}P_{n_i}$ (see Figure~\ref{fig:noterminalsTmp}).
    In particular, the path $P_{n_{k+1}}$ is distinct from $P_i$ for every $i\leq k$.
    At the end of this recursive process, $\{aP_{n_k}v_{n_k}\}_{k\in\mathbb{N}}$ and $\{v_{n_k}Q_{n_k}c\}_{k\in\mathbb{N}}$ define two systems, each consisting of pairwise edge-disjoint and order-compatible paths (because so are $\mathcal{P}$ and $\cQ$ respectively).
    
    Also, by definition of $v_{n_k}$, it follows that $(aP_{n_k}v_{n_k})\cap (v_{n_j}Q_{n_j}c) = \{v_{n_k}\}$ if $k=j$ and $(aP_{n_k}v_{n_k})\cap (v_{n_j}Q_{n_j}c) = \emptyset$ if $k\neq j$.
    Hence, \cref{lem:Concatenation} shows that $\mathcal{P}_{a,c}:=\{aP_{n_k}v_{n_k}Q_{n_k}c\}_{k\in\mathbb{N}}$ is a system of pairwise edge-disjoint and order-compatible paths connecting $a$ and $c$.
\end{proof}

Thus, for the remainder of the proof, we may suppose that the hypotheses of the previous two lemmas are not satisfied.
Then, for each $n\in\mathbb{N}$, we shall construct two infinite subfamilies $\mathcal{P}_n\subseteq \mathcal{P}_{a,b}$ and $\mathcal{Q}_n\subseteq \mathcal{P}_{b,c}$ that admit a $(\mathcal{P}_n,\mathcal{Q}_n)-$terminal $w_n$.
Moreover, we will ensure inductively that the following conditions are met:

\begin{enumerate}
    \item \label{item:ChoiceTerminal} $\mathcal{P}_{n+1}\subseteq \mathcal{P}_n$ and $\mathcal{Q}_{n+1}\subseteq \mathcal{Q}_n$ for every $n\in\mathbb{N}$. 
    In particular, $w_n\in Q$ for every $Q\in \mathcal{Q}_m$ and $m\geq n$;
    \item\label{item:ProgressingTerminal} $w_n$ belongs to finitely many paths in $\mathcal{P}_n$ and $w_n\notin \bigcup\mathcal{P}_{n+1}$, (so that $w_n$ is not a $(\cP_{n+1},\mathcal{Q}_{n+1}){-}$terminal).
\end{enumerate}

For the base case $n = 0$, let $\mathcal{P}_0\subseteq \mathcal{P}_{a,b}$ and $\mathcal{Q}_0\subseteq \mathcal{P}_{b,c}$ be two infinite families admitting a $(\mathcal{P}_0,\mathcal{Q}_0)$-terminal $w_0\in V(G)$, which exist since the hypothesis of \cref{lem:NoTerminals} is not satisfied. 
Similarly, since the hypothesis of \cref{lem:InfiniteHitting} is not verified, the terminal $w_0$ belongs to only finitely many paths in $\mathcal{P}_0$.

Suppose now that $\mathcal{P}_n$ and $\mathcal{Q}_n$ are defined for some $n\in\mathbb{N}$.
By \cref{item:ProgressingTerminal}, the set $\mathcal{P}_{n+1}'\subseteq \mathcal{P}_n$ of paths not containing $w_n$ is still infinite. 
Therefore, again because the hypothesis of \cref{lem:NoTerminals} is not satisfied, there are infinite subfamilies $\mathcal{P}_{n+1}\subseteq \mathcal{P}_{n+1}'$ and $\mathcal{Q}_{n+1}\subseteq \mathcal{Q}_n$ which admit a $(\mathcal{P}_{n+1},\mathcal{Q}_{n+1}){-}$terminal $w_{n+1}$.
And again due to the hypothesis of \cref{lem:InfiniteHitting} not applying, there are only finitely many paths in $\mathcal{P}_{n+1}$ that contain $w_{n+1}$.
In addition, $w_n\notin \bigcup\mathcal{P}_{n+1}$ by the choice of $\mathcal{P}_{n+1}'$ and since $\mathcal{P}_{n+1}\subseteq \mathcal{P}_{n+1}'$.

At the end of this inductive process, let $\Sigma_0:=\{cQw_0: Q\in\mathcal{Q}_0\}$ and $\Sigma_{n+1}:=\{w_{n}Qw_{n+1}:Q\in \mathcal{Q}_{n+1}\}$ for each $n\in\mathbb{N}$. 
In addition, denoting $w_{-1}:=c$, we will call $w_{n-1}$ the \defn{cut-vertex} of $\Sigma_n$.
Note that $\{ \Sigma_n \}_{n \in \mathbb{N}}$ is a family of well-defined paths because -- once $w_{n+1}$ is a $(\mathcal{P}_{n+1},\mathcal{Q}_{n+1}){-}$terminal -- we have $w_n,w_{n+1}\in Q$ for every $Q\in \mathcal{Q}_{n+1}$. 
In particular, since $\mathcal{Q}_n$ consists of infinitely many pairwise order-compatible paths for each $n\in\mathbb{N}$, we obtain the following observation:

\begin{lemma}\label{obs:UniqueInterval}
    For every pair of integers $0\leq m<n$ such that $(\bigcup\Sigma_m)\cap(\bigcup\Sigma_n)\neq \emptyset$, we have $n=m+1$ and $(\bigcup \Sigma_m)\cap(\bigcup \Sigma_{m+1}) = \{w_m\}$.
    
    Moreover, if $X$ and $X'$ are two $c{-}w_n$ paths such that $w_{i-1}Xw_{i},w_{i-1}X'w_{i}\in \Sigma_i$ and $w_{i-1}Xw_i\neq w_{i-1}X'w_i$ for every $0\leq i \leq n$, then $X$ and $X'$ are edge-disjoint and order-compatible. 
\end{lemma}
\begin{proof}[Proof of \cref{obs:UniqueInterval}]
    Indeed, if $v\in (\bigcup\Sigma_m)\cap (\bigcup\Sigma_n)$, there are paths $Q\in \mathcal{Q}_m$ and $Q'\in \mathcal{Q}_n$ containing $v$ and such that $v\in (w_{m-1}Qw_{m})\cap (w_{n-1}Q'w_{n})$. 
    In other words, when fixing the orientation as paths starting in $c$, the path $Q$ first traverses $v$ and then $w_{m}$, while $Q'$ first traverses $w_{n-1}$ and then $v$.   
    Since $Q$ and $Q'$ are order-compatible, this is only possible when $v = w_m$ and, thus, $n=m+1$ because $\{w_i\}_{i\in \mathbb{N}}$ defines infinitely many distinct vertices (by item \cref{item:ProgressingTerminal}).

    In order to prove the `Moreover'-part of the statement, fix now two $c{-}w_n$ paths $X$ and $X'$ such that $w_{i-1}Xw_{i},w_{i-1}X'w_{i}\in \Sigma_i$ and $w_{i-1}Xw_i\neq w_{i-1}X'w_i$ for every $i\in [n]$. 
    In particular, since every edge of $X$ has both endpoints in $\Sigma_i$ for some $i\in [n]$, the edge-disjointness between $X$ and $X'$ follows from the fact that $w_{i-1}Xw_i$ and $w_{i-1}X'w_i$ are subpaths of two distinct paths from the edge-disjoint family $\mathcal{Q}_{i}$, by definition of $\Sigma_i$.

    Now, aiming to prove that $X$ and $X'$ are order-compatible, fix two distinct vertices $x,x'\in X\cap X'$.
    Then, choose indices $i,i'\in [n]$ such that $x\in \bigcup\Sigma_i$, $x'\in \bigcup \Sigma_{i'}$ and $|i'-i|$ is minimized with these properties. 
    In particular, the previous paragraphs just concluded that $x\in w_{i-1}Xw_i\cap w_{i-1}X'w_i$ and $x'\in w_{i'-1}Xw_{i'}\cap w_{i'-1}X'w_{i'}$. 
    If $i=i'$, then the segments $w_{i-1}Xw_{i}$ and $w_{i-1}X'w_i$ are distinct elements of $\Sigma_i$ by assumption, and hence there are two different order-compatible paths $Q,Q' \in \mathcal{Q}_{i}$ that contain $w_{i-1}Xw_{i}$ and $w_{i-1}X'w_i$, respectively. 
    As such, $X$ and $X'$ both traverse the vertices $x$ and $x'$ in the same order.

    If $i\neq i'$, and thus assuming $i<i'$ without loss of generality, it follows from the minimality of $|i'-i|>0$ that $x\neq w_{i'-1}$: after all, $w_{i'-1}\in \Sigma_{i'}$ and we could have chosen $|i-i'|$ to be zero if $x=w_{i'-1}$. 
    Similarly, $x'\notin \Sigma_{i'-1}$ by the same minimality argument and since $0\leq |i'-1-i|<|i'-i|$.
    Therefore, when orienting both $X$ and $X'$ as paths starting at $c$, it follows that they traverse first $x$, then $w_{i'-1}$ and $x'$ after that. 
    In other words, $X$ and $X'$ are order compatible paths.
    \end{proof}

Now put $\Sigma:=\bigcup_{n\in \mathbb{N}}\Sigma_n$.
Due to items \cref{item:ChoiceTerminal} and \cref{item:ProgressingTerminal} in the definition of the pairs $\{(\mathcal{P}_n,\mathcal{Q}_n)\}_{n\in\mathbb{N}}$, the following claim is verified:

\begin{claim}\label{claim:choiceUj}
    For every $k\in \mathbb{N}$, fix a path $P_k\in \mathcal{P}_k$ containing $w_k$ (which exists since this is a $(\mathcal{P}_k,\mathcal{Q}_k){-}$terminal).
    Then, there is $n_k\geq k$ such that, for some $u_k\in \bigcup \Sigma_{n_k}$, we have $u_k\in P_k$ and $(aP_{k}u_k)\cap \bigcup \Sigma = \{u_k\}$. 
\end{claim}
\begin{proof}[Proof of \cref{claim:choiceUj}]
    Once $w_k\in P_k\cap \bigcup \Sigma$, there is trivially a vertex $u_k$ from $\bigcup \Sigma$ which is closest to $a$ in $P_k$, thus satisfying $(aP_ku_k)\cap \bigcup \Sigma = \{u_k\}$.
    In this case, let $n_k\in \mathbb{N}$ be the corresponding index such that $u_k\in \Sigma_{n_k}$ and fix a path $Q\in \mathcal{Q}_{n_k}$ containing $u_k$.
    Then, suppose $n_k<k$ for a contradiction.
    Since $u_k\in P_k\cap Q$ and $w_{n_k}$ is a $(\mathcal{P}_{n_k},\mathcal{Q}_{n_k}){-}$terminal, we should have $u_k = w_{n_k}$ because $P_k\in \mathcal{P}_k\subseteq \mathcal{P}_{n_k}$.
    On the other hand, also due to the inequality $n_k<k$, item \cref{item:ProgressingTerminal} from the construction of the cut-vertices would imply that $w_{n_k}\notin P_k$ and, thus, contradict the choice of $u_k$.
    Therefore, we must have $n_k\geq k$.
\end{proof}

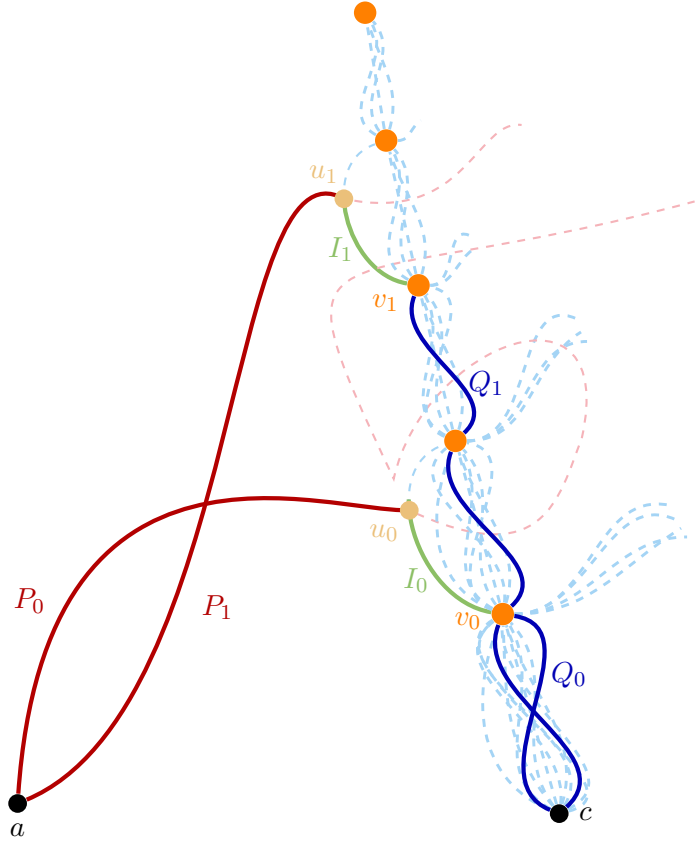
\begin{figure}[h]
    \centering
      \begin{tikzpicture}[yscale=0.9,
    vertex/.style={circle, fill=black, inner sep=2.5pt},
    setW/.style={circle, fill=orange, inner sep=3pt},
    setU/.style={circle, fill=lucasOrange, inner sep=2.5pt},
    bluepath/.style={lucasBlue, line width=1.1pt},
    redpath/.style={red!70!black, line width=1.1pt}
]

\node[vertex, fill=black,label=below:$a$] (a) at (-8,2.6) {};

\node[vertex, fill=black, label=right:$c$]  (c)  at (-0.8,2.45) {};
\node[setW] (v0) at (-1.55,5.4) {};
\node[setW] (v1) at (-2.18,7.96) {};
\node[setW] (v2) at (-2.67,10.26) {};
\node[setW] (v3) at (-3.1,12.4) {};
\node[setW] (v4) at (-3.38,14.29) {};


\draw[bluepath, thick, dashed] (v2) to[out=170, in=-160] coordinate[pos=0.6] (u1) (v3);
\draw[lucasDarkGreen, ultra thick, dash pattern=on 1.6cm off 100cm, dash phase=0cm] (v2) to[out=170, in=-160] (v3);
\draw[bluepath, thick, dashed] (v0) to[out=170, in=-160] coordinate[pos=0.6] (u2) (v1);
\draw[lucasDarkGreen, ultra thick, dash pattern=on 2cm off 100cm, dash phase=0cm] (v0) to[out=170, in=-160] (v1);


\draw[bluepath,dashed] (v0) to[out=0, in=-180] (0.9,6.54);
\draw[bluepath,dashed] (v0) to[out=0, in=-145] (0.82,6.67);
\draw[bluepath,dashed] (v0) to[out=0, in=154] (0.74,6.8);
\draw[bluepath,dashed] (v0) to[out=0, in=168] (0.66,7.01);

\draw[bluepath,dashed] (v1) to[out=0, in=-180] (-0.43,9.43);
\draw[bluepath,dashed] (v1) to[out=0, in=-145] (-0.51,9.57);
\draw[bluepath,dashed] (v1) to[out=0, in=154] (-0.58,9.77);

\draw[bluepath,dashed] (v2) to[out=0, in=-145] (-1.92,10.81);
\draw[bluepath,dashed] (v2) to[out=0, in=154] (-2,11);

\draw[bluepath,dashed] (v3) to[out=0, in=-145] (-2.64,12.7);

\draw[lucasLightRed, thick, dashed]
(a) .. controls (-5,4) and (-5.23,12.51) .. (u1) .. controls (-2.03,11.15) and (-1.9,12.72) .. (-1.3,12.63);
\node[left] at (-5,5.5) {\textcolor{red!70!black}{$P_1$}};
\draw[lucasLightRed, thick, dashed]
(a) .. controls (-7.63,8.5) and (-3.91,6.9) .. (u2) .. controls (-1,6) and (-0.81,7.1) .. (-0.5,8)
    .. controls (0,10.5) and (-2.6,9.4)
  .. (-3,7.4)
    .. controls (-4.7,11.4) and (-4,10)
  .. (1,11.5);
\node[left] at (-7.5,5.6) {\textcolor{red!70!black}{$P_0$}};
\draw[redpath, ultra thick]
(a) .. controls (-5,4) and (-5.23,12.51) .. (u1);
\draw[redpath, ultra thick]
(a) .. controls (-7.63,8.5) and (-3.91,6.9) .. (u2);

\foreach \A/\B in {c/v0, v0/v1, v1/v2, v2/v3, v3/v4}
{
  \draw[bluepath,dashed] (\A) to[out=149, in=-61] (\B);
    \draw[bluepath,dashed] (\A) to[out =75, in=-74] (\B);
    \draw[bluepath,dashed] (\A) to[out =124, in=-50] (\B);
  \draw[bluepath,dashed] (\A) -- (\B);
}

\draw[bluepath,dashed] (c) to[out=156, in=-110] (v0);
\draw[bluepath,dashed] (c) to[out=127, in=-88] (v0);
\draw[bluepath,dashed] (c) to[out=29, in=-93] (v0);
\draw[bluepath,dashed] (c) to[out=16, in=-147] (v0);
\draw[bluepath,dashed] (c) to[out=60, in=-140] (v0);
\draw[bluepath,dashed] (c) to[out=83, in=-139] (v0);
\draw[blue!70!black, ultra thick] (c) to[out=45, in=-110] (v0);
\draw[blue!70!black, ultra thick] (c) to[out=156, in=-10] (v0);

\draw[bluepath,dashed] (v0) to[out=83, in=-139] (v1);
\draw[bluepath,dashed] (v0) to[out=156, in=-110] (v1);
\draw[bluepath,dashed] (v0) to[out=156, in=-10] (v1);
\draw[blue!70!black, ultra thick] (v0) to[out=45, in=-110] (v1);

\draw[bluepath,dashed] (v1) to[out=156, in=-10] (v2);
\draw[blue!70!black, ultra thick] (v1) to[out=45, in=-110] (v2);

\node[setU] at (u1) {};
\node[setU] at (u2) {};
\node[left] at (-2.8,6.6) {\textcolor{lucasOrange}{$u_0$}};
\node[left] at (-3.6,11.9) {\textcolor{lucasOrange}{$u_1$}};
\node[left] at (-2.4,5.9) {\textcolor{lucasDarkGreen}{$I_0$}};
\node[left] at (-3.4,10.8) {\textcolor{lucasDarkGreen}{$I_1$}};
\node[left] at (-0.3,4.5) {\textcolor{blue!70!black}{$Q_0$}};
\node[left] at (-1.4,8.8) {\textcolor{blue!70!black}{$Q_1$}};
\node[left] at (-1.7,5.3) {\textcolor{orange}{$v_0$}};
\node[left] at (-2.8,10) {\textcolor{orange}{$v_1$}};

\end{tikzpicture}
    \caption{Construction of the paths $X_0,X_1,\ldots$, where the orange points correspond to the sequence of cut-vertices $\{w_n\}_{n\in\mathbb{N}}$.}
    \label{fig:omegatransitive}
\end{figure}

Now, fix the paths $\{P_k\}_{k\in\mathbb{N}}$ and vertices $\{u_k\}_{k\in\mathbb{N}}$ as in \cref{claim:choiceUj}.
After passing to a suitable subsequence, we may additionally assume that $n_0\geq 1$ and $n_{k+1}\geq n_{k}+2$ for every $k\in\mathbb{N}$.
By \cref{obs:UniqueInterval}, this means that $(\bigcup\Sigma_{n_k})\cap (\bigcup\Sigma_{n_j}) = \emptyset$ for two distinct $j,k\in \mathbb{N}$. 
Then, for each $k\in\mathbb{N}$, we fix the cut-vertex $v_k:=w_{n_k-1}$ of $\Sigma_{n_k}$ as well as a path $Q_k'\in \mathcal{Q}_{n_k}$ such that $I_k:=u_kQ_k'v_k$ is a segment of some path in $\Sigma_{n_k}$.
In the case where $u_k=v_k=w_{n_k-1}$, the segment $I_k$ consists of the single vertex $u_k$, regardless of the choice of $Q_k'$ (see Figure~\ref{fig:omegatransitive}).

By induction on $k$, we also fix $Q_k\in \mathcal{Q}_{n_k-1}\setminus \{Q_j,Q_j': j < k\}$. 
Note that $v_k=w_{n_k-1} \in Q_k$, and so $X_k:=aP_ku_kI_kv_kQ_kc$ is an $a{-}c$ path.

It remains to prove that $\mathcal{P}_{a,c}:=\{X_k\}_{k\in\mathbb{N}}$ defines a system of edge-disjoint and pairwise order-compatible $a-c$ paths.
Aiming to rely on \cref{lem:Concatenation} for that purpose, we consider the following subpaths given some $k \in \mathbb{N}$:

\begin{itemize}
    \item $X_{k}^-:=aX_{k}u_{k} = aP_ku_k$ and 
    \item $X_{k}^+:=u_{k}X_{k}c = u_k I_k v_{k}Q_{k}c$.
\end{itemize}

Given $j<k$, it follows that $X_k^-$ and $X_j^+$ are vertex-disjoint: 
Indeed, $X_j^+ \subseteq \bigcup_{i=0}^{n_j} \Sigma_i$ by definition and $X_k^- \cap \bigcup \Sigma = \{u_k\}$ by the definition of $u_k$ as in \cref{claim:choiceUj}, while $u_k \in \bigcup \Sigma_{n_k}$ and $\Sigma_{n_k} \cap \Sigma_{n_i}=\emptyset$ for every $0 \leq i \leq n_j$.
On the other hand, the paths $X_k^+$ and $X_j^-$ intersect in at most $u_j$ since $(aP_ju_j) \cap \Sigma = \{ u_j \}$ and $X_k^+ \subseteq \Sigma$ by definition.

Hence, it remains to prove that the pairs $X_j^-$ and $X_k^-$ as well as $X_k^+$ and $X_j^+$ are pairwise edge-disjoint and order-compatible, respectively.
For the pair $X_j^-$ and $X_k^-$ this follows by the fact that both paths define subpaths of distinct paths from $\cP_{a,b}$.
\begin{claim}
    $X_k^+$ and $X_j^+$ are edge-disjoint and order-compatible.
\end{claim}
\begin{proof}
    It holds that $X_j^+ \subseteq cQ_jv_jQ'_j w_{n_j}\subseteq \bigcup_{i=0}^{n_j}\Sigma_i$ and $X_k^+\cap \bigcup_{i=0}^{n_j}\Sigma_i = cQ_kw_{n_j}$ by construction and \cref{obs:UniqueInterval}.
    Hence, it suffices to conclude that the segments $cQ_jv_jQ_j'w_{n_j}$ and $cQ_kw_{n_j}$ are edge-disjoint and order-compatible.
    However, this follows from the 'Moreover'-part of \cref{obs:UniqueInterval} and the choice of the distinct edge-disjoint paths $Q_j,Q_j'$ and $Q_k$.
    After all, $w_{i-1}Q_jw_i\neq w_{i-1}Q_j'w_i$, $w_{i-1}Q_jw_i\neq w_{i-1}Q_kw_i$ and $w_{i-1}Q_j'w_i\neq w_{i-1}Q_kw_i$ for every $i\in [n_j]$, as well as $Q_j,Q_j',Q_k\in \mathcal{Q}_{n_j-1}$ because $\mathcal{Q}_{n_k-1}\subseteq \mathcal{Q}_{n_j}\subseteq \mathcal{Q}_{n_j-1}$.
\end{proof}

\section{Order compatibility is transitive - the uncountable case}
\label{sec_ordercomp_uncble}

We are now ready to complete the proof of \cref{thm_main2} announced in the introduction, which we repeat here for convenience.

\Transitive*

\begin{proof}
    If $\delta$ is infinite of uncountable cofinality, this follows from the fact that $ a \sim_\delta b$ if and only if $ a \approx_\delta b$ (\cref{cor_main}), and the fact that $\sim_\delta$ is transitive by Menger's theorem. The same argument applies if  $\delta$ is finite.
    Next, if $\delta = \omega$ this follows from the results in the previous \cref{sec_ordercomp_cble}.

    Thus, it remains to consider the case where $\delta$ is a singular cardinal of countable cofinality.
    Let $(\delta_{i})_{i \in \mathbb{N}}$ be an increasing sequence of regular cardinals converging towards $\delta$ with $\omega<\delta_{0}$.

    Then consider pairwise distinct vertices $a,b,c \in V(G)$ such that there exists an $a{-}b$ path system $\mathcal{P}_{a,b}$ as well as a $b{-}c$ path system $\mathcal{P}_{b,c}$ of $\delta$ many pairwise order-compatible $a{-}b$ paths and $b{-}c$ paths, respectively.
    For each $x \in \{ (a,b), (b,c) \}$ we partition $\mathcal{P}_{x}=\dot\bigcup_{i \in \N} \mathcal{P}^i_{x}$ into a disjoint union of subsets $\mathcal{P}^i_{x}$ ($i \in \N$) of size $|\mathcal{P}^i_{x}|=\delta_i$. 
    Moreover, since all $\delta_i$ are regular uncountable, we can assume that individually every $\mathcal{P}^i_{x}$ consists of paths of bounded length $\leq p_i \in \N$, as in the proof of \cref{cor_main}.

    By \cref{lem:regularCardCase}, for each $i \in \N$ there exists a sequence $P^i_x: a=t^i_0 t^i_1 \dots t^i_{k_i}=b$ of vertices from paths in $\mathcal{P}^i_x$ with $\kappa_V(t_m, t_{m+1})\geq \delta_i$ for every $m \in [k_i-1]$, together with a subfamily $\mathcal{Q}^{i}_x \subseteq \mathcal{P}^{i}_x$ of size $\delta_i$, such that every path in $\mathcal{Q}^{i}_x$ traverses every vertex on $P^i_x$ and in exactly the order induced by $P^i_x$.
    As a consequence, any two finite sequences $P^i_x$ and $P^j_x$ are order-compatible for any $i<j \in \mathbb{N}$: Picking some path $P_1 \in \mathcal{Q}^i_x$ and $P_2 \in \mathcal{Q}^j_x$, then $P_1,P_2 \in \cP_x$ are order-compatible by assumption, and both traverse every vertex on $P^i_x$, $P^j_x$ in the order given by $P^i_x$, $P^j_x$, respectively.

    We now view each finite sequence $P^i_x$ as a path, that is, as a simple graph with edges $E(P^i_{x}) = \{t^i_m t^i_{m+1} \colon m \in [k_i-1]\}$. We then consider an auxiliary multigraph $H$ with $V(H) = V(G)$ and with edge set $E(H)=\dot{\bigcup}_{i \in \mathbb{N}}(E(P^i_{a,b}) \, \dot\cup \, E(P^i_{b,c}))$.
    Then clearly, there exist $\omega$ many pairwise edge-disjoint and order-compatible paths from $a$ to $b$, and also from $b$ to $c$ in $H$. 
    Hence, from the results in the previous \cref{sec_ordercomp_cble} we conclude that there exist $\omega$ many pairwise edge-disjoint and order-compatible $a{-}c$ paths $(Q_i)_{i \in \mathbb{N}}$ in $H$.

    The proof will finish by an application of \cref{lem:auxilaryGraph}. 
    For this, it suffices to prove that for every $j \in \mathbb{N}$ there exists some $k \in \mathbb{N}$ such that $\kappa^G_V(x,y) \geq \delta_j$, for every edge $xy \in E(Q_k)$.
    Given some $j \in \mathbb{N}$, the set $\dot{\bigcup}_{i<j}\big(E(P^i_{a,b}) \, \dot{\cup} \,  E(P^i_{b,c}) \big)$ has finite size.
    Thus, there exists a path $Q_k$ edge-disjoint from $\dot{\bigcup}_{i<j}\big(E(P^i_{a,b}) \, \dot{\cup} \,  E(P^i_{b,c})\big)$ and hence only using edges from paths $P^i_{a,b}$ or $P^i_{b,c}$, where $i \geq j$.
    These paths $P^i_{a,b}$ or $P^i_{b,c}$ with $i \geq j$ witness that $\kappa^G_V(x,y)\geq \delta_j$ for every $xy \in E(Q_k)$.
\end{proof}

\bibliographystyle{amsplain}
\bibliography{collective.bib}

\end{document}